\documentclass{amsart}
\usepackage{tikz}
\usetikzlibrary{shapes.geometric}
\usepackage[center]{caption}
\usepackage{xcolor}
\usepackage{amssymb,latexsym,amsmath,extarrows}
\usepackage{graphicx,mathrsfs,comment}
\usepackage{hyperref,url}
\usepackage{pict2e}

\usepackage{amstext}
\usepackage{bbm}

\numberwithin{equation}{section}

\usepackage{esint}

\newtheorem{theorem}{Theorem}[section]
\newtheorem{lemma}[theorem]{Lemma}

\newtheorem{proposition}[theorem]{Proposition}

\newtheorem{corollary}[theorem]{Corollary}

\makeatletter
\newcommand{\barredsum}{%
  \DOTSB\mathop{\mathpalette\@barredsum\relax}\slimits@
}
\newcommand{\@barredsum}[2]{%
  \begingroup
  \sbox\z@{$#1\sum$}%
  \setlength{\unitlength}{\dimexpr2pt+\ht\z@+\dp\z@\relax}%
  \@barredsumthickness{#1}%
  \vphantom{\@barredsumbar}%
  \ooalign{$\m@th#1\sum$\cr\hidewidth$#1\@barredsumbar$\hidewidth\cr}%
  \endgroup
}
\newcommand{\@barredsumbar}{%
  \vcenter{\hbox{\begin{picture}(0,1)\roundcap\Line(0,0)(0,1)\end{picture}}}%
}
\newcommand{\@barredsumthickness}[1]{
  \linethickness{%
    1.25\fontdimen8
      \ifx#1\displaystyle\textfont\else
      \ifx#1\textstyle\textfont\else
      \ifx#1\scriptstyle\scriptfont\else
      \scriptscriptfont\fi\fi\fi 3
  }%
}
\makeatother

\newcommand{\e}{\epsilon}

\newcommand{\ZR}{\mathbb{R}}

\newcommand{\ZZ}{\mathbb{Z}}

\newcommand{\cT}{{\mathcal T}}

\newcommand{\R}{\mathbb{R}}

\begin{document}

\title[Energy bound for points on algebraic surfaces]{Near diagonal additive energy bound for points on algebraic surfaces}

\author{Yifan Jing}
\address{Department of Mathematics, Ohio State University, Columbus, OH, USA}
\email{jing.245@osu.edu}

\author{Shukun Wu}
\address{Department of Mathematics, Indiana University Bloomington, Bloomington, IN, USA}
\email{shukwu@iu.edu}

\date{}
\begin{abstract} 
Let $F:\ZR^3\to\ZR$ be a polynomial that is irreducible over $\ZR$ with $\deg F\geq2$.
We prove that, for any finite $X\subset Z(F)$ that does not concentrate on affine lines,
\[
    E(X)=\#\{(a,b,c,d)\in X^4: a+b=c+d\}\ll_{\deg F,\,\epsilon}(\# X)^{2+\epsilon}.
\]
In particular, this answers a question of Bourgain and Demeter concerning finite subsets of the unit sphere.

\end{abstract}

\maketitle

\section{Introduction}

For finite sets $X,Y\subset\ZR^n$, define
\[
    d_{X-Y}(t)=\#\{(x,y)\in X\times Y: x-y=t\}, \qquad E(X,Y)=\sum_s d_{X-Y}(s)^2.
\]
For convenience, set $d_X=d_{X-X}$ and $E(X)=E(X,X)$.
A useful identity is
\begin{equation}
\label{useful-identity}
    E(X,Y)=\sum_t d_X(t)d_Y(t).
\end{equation}
For a polynomial $P:\ZR^n\to\ZR$, write
\[
    Z(P)=\{x\in\ZR^n:P(x)=0\}.
\]
For two polynomials $P,Q$, write $Z(P,Q)=Z(P)\cap Z(Q)$.

\begin{theorem}
\label{main-thm}
Let $F:\ZR^3\to\ZR$ be a polynomial that is irreducible over $\ZR$.
Define 
\begin{equation}
\label{quantifier}
    \Lambda_{F}(X) = \max\left\{1, \,\sup_{\ell \text{ an affine line in }Z(F)}\#(X\cap\ell)\right\}
\end{equation}
to quantify the extent to which $X$ concentrates on lines in $Z(F)$.

If $\deg F\geq2$, then for any $\e>0$, there exists a constant $C_\e$ depending only on $\e$ and the degree of $F$ such that
\begin{equation}
\label{main-esti}
    E(X)\leq C_{\e}\Lambda_{F}(X)(\#X)^{2+\e}
\end{equation}
for all finite $X\subset Z(F)$.
\end{theorem} 

Note that whenever $\Lambda_{F}(X)\ll 1$,  \eqref{main-esti} yields
\begin{equation}
\label{diagonal-esti}
    E(X)\ll_{\e}(\#X)^{2+\e}.
\end{equation}
This estimate is essentially sharp, since the diagonal solutions alone contribute $\asymp (\#X)^2$ to $E(X)$.
Surfaces such as the paraboloid and the unit sphere are convex and hence satisfy $\Lambda_F(X)\ll 1$. 
Bourgain and Demeter \cite{Bourgain-Demeter} asked whether the near-diagonal estimate \eqref{main-esti} holds when $Z(F)$ is the unit sphere, and Theorem \ref{main-thm} answers this question affirmatively.

\medskip

By contrast, if $Z(F)$ contains a line, one may take $X$ to be an arithmetic progression on that line, in which case $E(X)$ is of order $(\#X)^3$. 
Familiar examples include the quadratic cone and the hyperbolic paraboloid: the former is ruled by a one-parameter family of lines through its vertex, whereas the latter is doubly ruled. An arithmetic progression along any one of these ruling lines therefore exhibits the cubic growth above.

The presence of such lines does not make Theorem $\ref{main-thm}$ invalid for these surfaces. The obstruction to a near-diagonal estimate is measured quantitatively by $\Lambda_F(X)$, rather than by the existence of lines on $Z(F)$. 
For example, if $X$ is contained in a quadratic cone or a hyperbolic paraboloid but meets every affine line contained in the surface in $O(1)$ points, then $\Lambda_F(X)\ll1$, and the near-diagonal estimate $\eqref{diagonal-esti}$ still follows. 
More generally, $\eqref{main-esti}$ quantifies the loss caused by concentration of $X$ along the ruling lines. 
Theorem~\ref{main-thm} shows that, for algebraic surfaces, concentration on affine lines is the only obstruction for the near diagonal bound. We also remark that the dependence on $\Lambda_F(X)$ in $\eqref{main-esti}$ is sharp.

Theorem \ref{main-thm} is also sharp in the following sense: for $k\geq 3$, one cannot expect the near-diagonal estimate $E_k(X)\ll_\epsilon |X|^{k+\epsilon}$ for any finite $X\subset Z(F)$ that does not concentrate on affine lines,
where
\[
    E_k(X)
    :=
    \#\{(x_1,\ldots,x_{2k})\in X^{2k}:
    \sum_{i=1}^kx_i
    =
    \sum_{j=k+1}^{2k}x_j\}.
\]
A typical example is the paraboloid $Z(x^2+y^2-z)$, where $E_k(X)$ can be as large as $(\#X)^{2k-2}$.
We collect some other examples in Section \ref{example-section}.

\medskip

A few results were known for particular surfaces.
When $Z(F)$ is the paraboloid, Theorem \ref{main-thm} was proved in \cite{Bourgain-Demeter} using the corner-counting result of Pach and Sharir \cite{Pach-Sharir}.
More precisely, it is proved that if $X\subset Z(x^2+y^2-z)$, then
\[
    E(X)\ll (\#X)^2\log (\#X),
\]
which is sharp up to a constant factor.
When $Z(F)$ is the unit sphere, Mudgal \cite{Mudgal-sphere} proved that $E(X)\ll_\e (\#X)^{2+2/9+\e}$.

The spherical case of Theorem \ref{main-thm} also yields a result for lattice points on large spheres that appears to be new in the sparse regime.
More precisely, let
\[
    \mathcal{S}_N
    =
    \bigl\{x\in\mathbb{R}^3:|x|=N\bigr\},
    \qquad
    X_N
    =
    \mathcal{S}_N\cap\mathbb{Z}^3.
\]
Since dilation by $N^{-1}$ maps $X_N$ into the unit sphere and preserves additive energy, Theorem \ref{main-thm} gives,  uniformly in $N$,
\begin{equation}
\label{lattice-count}
    E(X_N)\ll_{\e}(\#X_N)^{2+\e}.
\end{equation}
In particular, this estimate appears to be new when the lattice is exceptionally sparse, for example when
\[
    \#X_N<N^\eta
\]
for arbitrarily small $\eta>0$.
This case is not effectively handled by the Bourgain--Demeter decoupling theorem \cite{Bourgain-Demeter}, which yields only $E(X_N)\ll_\e N^\e(\# X_N)^2$, a bound weaker than \eqref{lattice-count}.
We remark that the number of lattice points on $\mathcal{S}_N$ fluctuates substantially with $N$: for almost all integer radii $N$, one has $\#(\mathcal{S}_N\cap\mathbb{Z}^3)=N^{1+o(1)}$, whereas, when $N$ is a power of $2$, the sphere contains exactly the six axial lattice points.

\smallskip 

More generally, let $Q$ be a nondegenerate quadratic polynomial with positive-definite quadratic part and negative minimum, so that $Z(Q)$ is a nonempty ellipsoid. Ellipsoids arise naturally in the study of eigenfunctions of the Laplacian on flat tori.
When $Z(Q)$ is the unit sphere, the corresponding torus is the standard torus $\ZR^3/\ZZ^3$. See \cite{GRP} for further discussion.

\bigskip

\subsection{Examples}
\label{example-section}
We only calculate $E_3$ for the examples.
In each example, the selected points lie in a smooth part of the surface with positive-definite second fundamental form.

\medskip 

{\bf 1. $Z(x^2+y^d-z)$.} For $d\ge 3$, consider the grid
\[
    X
    =
    \left\{
    \left(\frac{i}{N^d},
    \frac{j}{N^2},
    \frac{i^2+j^d}{N^{2d}}
    \right):
    1\le i\le N^d,\quad
    1\le j\le N^2
    \right\}.
\]
Then $\# X\asymp N^{d+2}$ and
\[
    \#(X+X+X)\ll N^dN^2N^{2d} = N^{3d+2}\ \Longrightarrow\ E_3(X)\gg (\#X)^{3+\frac{4}{d+2}}.
\] 

\medskip 

{\bf 2. $Z((y+x^d)^2-z)$.} 
For $d\ge 2$, choose
\[
    X=\left\{
    \left(
    \frac{i}{N},
    \frac{j-i^d}{N^d},
    \frac{j^2}{N^{2d}}
    \right) :N\leq i<2N,\quad N^d\leq j<2N^d \right\}.
\]
Then $\#X \asymp N^{d+1}$
and
\[
    \#(X+X+X)\ll  N^{3d+1}\ \Longrightarrow\ E_3(X)\gg (\#X)^{3+\frac{2}{d+1}}.
\]

\medskip

{\bf 3. $Z(x^3+y^3-z)$.} Choose
\[
    X = \left\{
    \left(
    \frac{i}{N},
    \frac{j}{N},
    \frac{i^3+j^3}{N^3}
    \right):
    N  \leq i,j\leq 2N \right\}.
\]
Then $\# X\asymp N^2$ and
\[
    \#(X+X+X)\ll N^5\ \Longrightarrow\ E_3(X)\gg (\#X)^{3+\frac{1}{2}}.
\]

\bigskip 

\subsection{Outline of the proof}
Define 
\[
    m_F:=\deg F.
\]
At a cost of $O(m_F)$, we may assume that $X\subset \Sigma$, where $\Sigma$ is the graph of a function $f(x,y)$ on an appropriate set $\mathcal G$ (see Lemma \ref{lem: graph cover}).
Our proof uses polynomial partitioning in the plane to realize the following iteration formula: 
For any natural number $D\geq2$, 
\begin{equation}
\label{iteration-intro}
    E(X)\ll_{m_F} \Lambda_{F}(X)D^2(\#X)^2+D^4\sup_{\substack{X'\subset X,\\
    \#X'\ll D^{-2}\#X}}E(X').
\end{equation}
With an appropriate choice of $D$, we can iterate \eqref{iteration-intro} to prove Theorem \ref{main-thm}.

\smallskip

To be more precise, let $\pi$ be the projection onto the $xy$-plane.
By running the polynomial partitioning in the plane (see Theorem \ref{thm:planar-partition}) for $\pi(X)$ and then lifting it back to the surface $\Sigma$, we obtain a polynomial $q$ of degree at most $D$ and the induced cell decomposition
\[
    \Sigma=\Omega\sqcup\big(\bigsqcup_j\Sigma_j\big)
\]
obeying the following
\begin{enumerate}
    \item $\#\{j: \Sigma_j\not=\varnothing\}\ll D^2$, and $\#X_j\ll D^{-2}\#X$, where $X_j:=X\cap \Sigma_j$.
    \item $\Omega$ is contained in a union of algebraic curves on $\Sigma$ with total degree $\ll_{m_F}D$.
\end{enumerate}

Let $Y=\sqcup X_j$ and $W=\Omega\cap X$ denote the cell and wall portion of $X$, respectively.
Note that for every $t\in\Sigma-\Sigma$ outside the exceptional set $\mathcal T_F$ given in \eqref{translation-stabilizer}, the set $\Gamma_t=Z(F(x), F(x-t))$ is an algebraic curve of degree $O_{m_F}(1)$, which encodes the pairs $(x,y)\in\Sigma^2$ such that $x-y=t$, as $\Gamma_t\supset\{x\in\Sigma: x-t\in\Sigma\}$.
Thus, by Theorem \ref{thm:curve-cutting}, there are $\ll_{m_F}D$ pairs of cells $(X_i, X_j)$ such that $d_{X_i-X_j}(t)\not=0$ (see \eqref{pairs-count}).
This implies that, for $t\not\in\cT_F$,
\begin{equation}
\label{sum-after-poly-intro}
    d_{Y}(t)^2=\Big(\sum_{i,j}d_{X_i-X_j}(t)\Big)^2\ll_{m_F}D\,\sum_{i,j}d_{X_i-X_j}(t)^2.
\end{equation}
Summing over $t$ with $d_{X}(t)>0$ so that
\begin{align}
\nonumber
    E(Y)\, &\ll_{m_F} \sum_{t\in\cT_F}d_Y(t)^2+D\,\sum_{t\not\in\cT_F}\sum_{i,j}d_{X_i-X_j}(t)^2\\ \label{sum-after-poly2-intro}
    &\ll  \Lambda_{F}(X)(\#X)^2+D\,\sum_{t\not\in\cT_F}\sum_{i,j}d_{X_i-X_j}(t)^2.
\end{align}
In the second inequality, we use Corollary \ref{cor:difference curves}.

Using the identity \eqref{useful-identity}, we have
\[
    D\,\sum_{t\not\in\cT_F}\sum_{i,j}d_{X_i-X_j}(t)^2\ll D\,\sum_t\Big(\sum_j d_{X_j}(t)\Big)^2.
\]
Similar to \eqref{sum-after-poly-intro}, we have
\[
    \sum_t\Big(\sum_j d_{X_j}(t)\Big)^2\ll \Lambda_{F}(X)(\#X)^2+D\,\sum_{t\not\in\cT_F}\sum_{i}d_{X_i}(t)^2.
\]
Plug these two estimates back into \eqref{sum-after-poly2-intro} to conclude
\[
    E(Y)\ll_{m_F} \Lambda_{F}(X)D(\#X)^2+D^4\sup_{\substack{X'\subset X,\\
    \#X'\ll D^{-2}\#X}}E(X').
\]

\medskip

Next, we consider the wall portion $W$ by finding an upper bound for 
\[
    E(X,W)=\sum_t d_X(t)d_W(t),
\] 
We only bound $d_W(t)$.
In Lemma \ref{lem: wall energy}, we show that $d_W(t)=O_{m_F}(D^2)$ for most $t$, while the remaining values of $t$ contribute a total of $\Lambda_{F}(X)D(\#X)^2$.
Thus, 
\[
    E(X,W)\ll_{m_F}\Lambda_{F}(X)D^2(\#X)^2.
\]

Since $X=Y\cup W$, combining the cellular contribution and the wall contribution, we end up with
\[
    E(X)\ll E(X,W)+E(Y)\ll_{m_F} \Lambda_{F}(X)D^2(\#X)^2+D^4\sup_{\substack{X'\subset X,\\
    \#X'\ll D^{-2}\#X}}E(X').
\]
This is \eqref{iteration-intro}.

\bigskip

\subsection{Communication with AI}
During the development of this work, the authors used  GPT-5.6 to explore a polynomial-partitioning approach to the unit-sphere case. 
The model suggested the initial structure of the spherical argument. 
The authors subsequently reconstructed and independently verified the argument, established its extension to general algebraic surfaces, and wrote the manuscript.

\bigskip

\noindent {\bf Notation.}
We write $A\ll B$ to mean that $|A|\leq C B$ for some constant $C>0$. 
A subscript indicates the parameters on which the implicit constant may depend; for example, $A\ll_{\varepsilon} B$ means that $C$ may depend on $\varepsilon$.
Unless otherwise stated, all implicit constants are absolute.
For a finite set $X$, we use $\# X$ to denote its cardinality.

\bigskip

\noindent {\bf Acknowledgements.} 
The first author is partially supported by the NSF grant DMS-2503063.
The second author is partially supported by the NSF grant DMS-2453583 and thanks Ciprian Demeter for helpful comments on the introduction.

\bigskip

\section{Preliminary results}

\subsection{Algebraic inputs}

In this subsection, we collect several algebraic results used in the proof. 
All are well known and are stated without proof.

\medskip

The first one is the Guth--Katz polynomial partitioning.

\begin{theorem}[\cite{Guth-Katz}, Theorem 4.1]\label{thm:planar-partition}
There is an absolute constant $\alpha\ge1$ such that, for every finite $Y\subset\R^2$ and every integer $D\ge1$, there is a nonzero polynomial $q\in\R[u,v]$ of degree at most $D$ for which every connected component of $\R^2\setminus Z(q)$ contains at most
\[
 \alpha\frac{\# Y}{D^2}
\]
points of $Y$, and $b_0(\R^2\setminus Z(q))\leq \alpha D^2$, where $b_0$ is the number of semialgebraically connected components.  
\end{theorem}

The next result is a generalized affine B\'ezout theorem over the field $\R$.
See \cite[Chapter~4]{BasuPollackRoy2006}. 
Here and throughout the remainder of the paper, if $C$ is an affine algebraic curve defined over $\R$, we write $C(\R)$ for its set of real points. 

\begin{theorem}[B\'ezout]\label{thm: Bezout}
Let $P,Q\in\R[x_1,x_2,x_3]$ be nonconstant and coprime, with degrees $d_1$ and $d_2$, respectively.
There is a possibly empty family of irreducible affine algebraic curves $C_1,\dots,C_r$ defined over $\R$ such that
\begin{equation}
\label{eq:bezout-decomposition}
    Z(P,Q)=\bigcup_{i=1}^r C_i(\R),
\end{equation}
and $\sum_{i=1}^r\deg C_i\leq d_1d_2$.
In particular, if $C_1$ and $C_2$ are distinct irreducible affine algebraic curves defined
over $\R$, then
\begin{equation}
\label{eq:curve-bezout}
    \#\bigl(C_1(\R)\cap C_2(\R)\bigr)  \leq(\deg C_1)(\deg C_2).
\end{equation}
\end{theorem}

We also use the following one-dimensional case of the sign-condition theorem to bound the number of pieces into which one algebraic curve is cut by another, see  \cite[Theorem 1.1]{BaroneBasu2012}.
\begin{theorem}
\label{thm:curve-cutting}
For every integer $e\geq1$, there is a constant $C_e$ with the following property.  
Let $P_1, P_2\in\R[x_1, x_2, x_3]$ be nonzero and coprime, with $\deg P_1, \deg P_2\leq e$, and let $H_1,H_2\in\R[x_1,x_2, x_3]$ have degrees at most $D$.  
Then
\begin{equation}
\nonumber
    b_0 \left( Z(P_1, P_2)\setminus \bigl( Z(H_1)\cup Z(H_2)\bigr)\right)
    \leq C_e(D+1),
\end{equation}
where $b_0$ is the number of semialgebraically connected components.
\end{theorem}

\bigskip 

\subsection{Technical lemmas}

Recall that we write $m_F=\deg F$. 
We define the translation stabilizer of $F$ by
\begin{equation}
\label{translation-stabilizer}
    \mathcal T_F:=\{t\in\R^3:F(x-t)=F(x)\text{ for every }x\in\R^3\}.
\end{equation}

\begin{lemma}
\label{difference-curve-lem}
Let $F:\ZR^3\to\ZR$ be a polynomial that is irreducible over $\R$ with $\deg F\geq2$, $Z(F)\neq\varnothing$, and let $\cT_F$ be given by \eqref{translation-stabilizer}.
Then the set $\mathcal T_F$ is a linear subspace of $\R^3$ of dimension at most one. 
Also, for $t\notin\mathcal T_F$, $F(x)$ and $F(x-t)$ are coprime, and
\begin{equation}
\label{eq:def-spatial-difference-curve}
    \Gamma_t:=Z\bigl(F(x),F(x-t)\bigr)
\end{equation}
is a union of irreducible affine algebraic curves $C_{1},\dots,C_{r_t}$ satisfying
\begin{equation}
\label{eq:spatial-difference-degree}
    \sum_{i=1}^{r_t}\deg C_{i}\leq m_F^2.
\end{equation}
Moreover, if $\mathcal T_F\neq\{0\}$, then every affine coset $x+\mathcal T_F$ through a point $x\in Z(F)$ is an affine line contained in $Z(F)$.  
\end{lemma}

\begin{proof}
Suppose first that $t\in\mathcal T_F$.  By the definition of $\mathcal T_F$, for any fixed $x$, the polynomial
\[
    g_x(\tau)=F(x-\tau t)
\]
is $1$-periodic. 
Indeed for each $n\in\mathbb Z$, $g_x(\tau + n ) = F(x-(\tau+n-1)t-t)=F(x-(\tau+n-1)t)$ as $t\in\mathcal T_F$. 
This implies that $g_x$ is a constant polynomial. 
Hence $F(x-\tau t)=F(x)$ for every $\tau\in\R$. 
It is straightforward to verify that $\mathcal T_F$ is closed under real scalar multiplication and addition, and hence it is a linear subspace.

Assume, for contradiction, that $\dim\mathcal T_F\geq2$.
Then $\mathcal T_F$ contains at least two linearly independent vectors. 
Make an invertible linear change of coordinates of $\R^3$ so that $F$ is invariant in the first two coordinate directions. 
Therefore, the first two corresponding directional derivatives of $F$ vanish, and we may write $F(y_1,y_2,y_3)=G(y_3)$ for a univariate polynomial $G$.
Since $Z(F)\not=\varnothing$, the univariate polynomial $G$ has a real root. 
Because $F$, and hence $G$, is irreducible over $\ZR$, this forces $G$ to be linear, contradicting $\deg F\geq 2$.
Therefore $\dim\mathcal T_F\leq1$.

\smallskip

Next, let $t\notin\mathcal T_F$. 
Assume first that $F(x)$ and $F(x-t)$ are not coprime. 
As they are both irreducible, by the degree comparison, we have $F(x-t)=cF(x)$ for some nonzero $c\in\R$. 
On the other hand, translation does not change the highest homogeneous part, so comparison of highest homogeneous terms gives $c=1$. 
This implies that $t\in\mathcal T_F$, which contradicts the assumption. 
Applying Bezout's theorem (Theorem~\ref{thm: Bezout}) we proved \eqref{eq:spatial-difference-degree}.

\smallskip 

Finally suppose that $\mathcal T_F\neq\{0\}$. As $\mathcal T_F$ is a linear subspace, for each $x\in Z(F)$, clearly $x+\mathcal T_F\subseteq Z(F)$. 
Indeed, for each nonzero $t\in \mathcal T_F$, we have $F(x-t)=F(x)=0$ whenever $x\in Z(F)$. 
\end{proof}

\begin{corollary}
\label{cor:difference curves}
Let $\mathcal T_F$ be as in Lemma \ref{difference-curve-lem}. 
Then every finite $X\subset Z(F)$ satisfies
\begin{equation}
\label{eq:stabilizer-energy}
    \sum_{t\in\mathcal T_F}d_{X}(t)^2\leq \Lambda_{F}(X)(\#X)^2.
\end{equation}
\end{corollary}

\begin{proof}
We may partition $X$ among the cosets of $\mathcal T_F$ that meet $X$, and write $X=\bigsqcup_L X_L$. 
Note that if $x-y=t\in \mathcal T_F$, then $x$ and $y$ live in the same $\mathcal T_F$ coset.
Therefore
\[
    \sum_{t\in\mathcal T_F}d_{X}(t)
    =\sum_L|X_L|^2
    \leq\Lambda_{F}(X)\sum_L|X_L|
    =\Lambda_{F}(X)(\#X).
\]
The above inequality is also valid when $\mathcal T_F=\{0\}$ with singleton cosets. 
Since $d_{X}(t)\leq(\#X)$ for every $t$, we obtain
\[
    \sum_{t\in\mathcal T_F}d_{X}(t)^2
    \leq (\#X)\sum_{t\in\mathcal T_F}d_X(t)
    \leq\Lambda_{F}(X)(\#X)^2. \qedhere
\]

\end{proof}

\medskip

\begin{lemma}
\label{lem: graph cover}
Let $F:\ZR^3\to\ZR$ be a polynomial that is irreducible over $\ZR$, and let $F_m$ be the homogeneous part of $F$ of degree $m:=m_F$.  
There is an invertible linear map $L:\R^3\to\R^3$ such that, with $\widetilde F(y):=F(Ly)$, the coefficient of $y_3^m$ in $\widetilde F$ is nonzero. 
Moreover, for every finite $X\subset  Z(F)$,
\begin{equation}
\label{eq:linear-invariance}
    E(L^{-1}X)=E(X), \qquad
    \Lambda_{\widetilde F}(L^{-1}X)=\Lambda_{F}(X).
\end{equation}
In these normalized coordinates given by $L$, there are sets $\Omega_1,\dots,\Omega_m\subset\R^2$ and functions $f_j:\Omega_j\to\R$ such that, with $\mathcal G_j=\{(u,f_j(u)):u\in\Omega_j\}$, one has the disjoint graph decomposition $Z(\widetilde F)=\bigsqcup_{j=1}^m\mathcal G_j$.
Consequently, every finite $X\subseteq  Z(\widetilde F)$ has a disjoint decomposition
\begin{equation}
\label{eq:fiber-rank-decomposition}
    X=\bigsqcup_{i=1}^K X_i,
\end{equation}
with $K\leq m$, and each $X_j$ is contained in a portion $\mathcal G_j$ of $ Z(\widetilde F)$ that is the graph of a function of $\vec{u}=(y_1,y_2)$. In particular, the projection $\pi(y_1,y_2,y_3)=(y_1,y_2)$ is a bijection on every $X_j$.
\end{lemma}

\begin{proof}
Since $F_m$ is nonzero, choose $e\in\R^3$ with $F_m(e)\neq0$, and choose $L\in\mathrm{GL}_3(\R)$ with $Le_3=e$, where $e_3$ is the vertical unit vector. 
The top homogeneous part of $\widetilde F=F\circ L$ is $F_m\circ L$, so the coefficient of $y_3^m$ is $F_m(Le_3)=F_m(e)\neq0$.

Since $L$ is invertible, it is a group isomorphism and hence a bijection and a Freiman isomorphism of all orders.
In particular, $a+b=c+d$ if and only if $L(a)+L(b)=L(c)+L(d)$.
This implies the first identity in \eqref{eq:linear-invariance}. 
The map $L$ also gives a bijection between affine lines in $ Z(\widetilde F)$ and affine lines in $Z(F)$, preserving the number of selected points. 
This proves the second identity in \eqref{eq:linear-invariance}.

\smallskip

Now for any fixed $u\in\R^2$, the univariate polynomial $g_u(z)=\widetilde F(u,z)$ has degree $m$, with the same nonzero leading coefficient for every $u$. 
Consider distinct real roots of $g_u(z)$
\[
    z_1(u)<\cdots<z_{r(u)}(u).
\]
Clearly, $r(u)\leq m$. For $1\leq j\leq m$, define
\[
    \Omega_j=\{u\in\R^2:r(u)\geq j\},
\]
and $f_j(u)=z_j(u)$ for $u\in\Omega_j$. 
Thus, every real zero of $\widetilde F$ belongs to exactly one of the graphs $\mathcal G_j$. 
For a finite set $X\subset Z(\widetilde F)$, take $X_j=X\cap\mathcal G_j$. 
This gives the disjoint decomposition \eqref{eq:fiber-rank-decomposition}, and each remaining $X_j$ lies in a graph portion of $ Z(\widetilde F)$.
\end{proof}

\medskip

Finally, we state a lemma that is useful to bound the contribution from the wall coming from the polynomial partitioning.

\begin{lemma}
\label{lem: wall energy} 
Let $F:\ZR^3\to\ZR$ be a polynomial that is irreducible over $\ZR$ and $\deg_z F\ge 1$. 
Let $q\in\R[u_1,u_2]$ be nonzero of degree $\leq D$,  and set $Q(\vec{u},z)=q(\vec{u})$ so that $Q$ is independent of $z$. 
Let $Y\subset \ZR^3$ and
\[
    W\subset Z(F,Q)
\]
be finite sets with $\#W,\#Y\leq N$. 
Then
\begin{equation}
\label{eq:wall-energy}
    E(W,Y)\leq C_{m_F}\Lambda_{F}(W)D^2N^2.
\end{equation}
\end{lemma}

\begin{proof}
If $q$ is constant, then it is a nonzero constant and $W\subset Z(Q)=\varnothing$, so the result is trivial.
Henceforth assume that $\deg q\ge1$, so $D\geq 1$.
Since $F$ depends nontrivially on $z$ and since $F$ is irreducible over $\R$, the polynomials $F$ and $Q$ are coprime.
Now B\'ezout's Theorem (Theorem~\ref{thm: Bezout}) implies that there are irreducible affine algebraic curves $C_1,\dots,C_r$ defined over $\R$ such that
\[
    Z(F,Q)=\bigcup_{i=1}^rC_i(\R).
\]
Denote by $R_0 = \sum_{i=1}^r\deg C_i$, so 
\[
    R_0\leq m_FD.
\]
Assign every point of $W$ to one curve containing it.  
After discarding empty pieces, this gives a disjoint decomposition
\[
    W=\bigsqcup_{i=1}^r W_i
\]
with $W_i\subseteq C_i(\R)$. 

\smallskip

Let $I$ be the set of indices such that for each $i\in I$, $C_i$ is an affine line. 
Define
\[
    W_{\mathrm{line}}=\bigsqcup_{i\in I}W_i.
\]
Since every such line is contained in $Z(F)$, we have $\#I\leq R_0$, and
\begin{equation}
\label{eq: wall line points}
    \#W_{\mathrm{line}}\leq R_0\Lambda_{ F}(W).
\end{equation}

\smallskip

For $h\in\R^3$, define the translation map $\tau_h(C):=C+h$.  
A pair $(w_i,w_j)\in W_i\times W_j$ with $w_i-w_j=h$ forces
\[
    w_i\in C_i(\R)\cap\tau_h(C_j)(\R).
\]
Thus, for a given $h$ and a pair $(C_i, C_j)$ with $C_i\neq\tau_h(C_j)$, Theorem~\ref{thm: Bezout} implies
\begin{equation}
\label{not-translation}
    \#\{(w_i, w_j)\in W_i\times W_j: w_i-w_j=h \}\leq (\deg C_i)(\deg C_j).
\end{equation} 

Next, consider a pair $(C_i,C_j)$ that is not a pair of affine lines. 
Then there is at most one vector $h$ such that $C_i=\tau_h(C_j)$. 
Indeed, if both $h$ and $h'$ satisfy $C_i=\tau_h(C_j)$ and $C_i=\tau_{h'}(C_j)$, then $C_i$ would be invariant under translation by the nonzero vector $a=h'-h$.
Iterating it gives $C_i = C_i + na$ for every $n\in\mathbb Z$. 
Thus, for any polynomial $P$ that vanishes on $C_i$ and for every $w\in C_i$, the univariate polynomial $g(x)=P(w+xa)$ vanishes on every integer, hence is identically $0$.
Hence $w+\R a\subset C_i(\R)$.
Since an irreducible affine curve containing an affine line must equal that line, $C_i(\ZR)$ is a line.
Similarly, $C_j(\ZR)$ is also a line, and this contradicts the choice of the pair $(C_i, C_j)$.

Let $T$ be the set of all such exceptional translations $h$ arising from ordered pairs $(C_i,C_j)$ that are not both affine lines and satisfy $C_i=\tau_h(C_j)$.
Then the aforementioned discussion implies
\[
    \#T\leq r^2\leq R_0^2.
\]
Let $d_*(h)$ denote the number of pairs $(w_i,w_j)$ in $W\times W$ such that $w_i-w_j=h$ and the associated component pair $(C_i,C_j)$, with $w_i\in C_i$ and $w_j\in C_j$, is not a pair of two affine lines.
For $h\notin T$, by \eqref{not-translation} and summing over the relevant ordered pairs, 
\[
    d_*(h)\leq \sum_{i,j}(\deg C_i)(\deg C_j)=\Big(\sum_{i=1}^r\deg C_i\Big)^2\leq R_0^2.
\]
For $h\in T$, we trivially have $d_*(h),d_{Y}(h)\leq N$. Since $\sum_hd_{Y}(h)=|Y|^2\leq N^2$, 
\begin{align}
\nonumber
    \sum_h d_*(h)d_{Y}(h)=&\sum_{h\not\in T} d_*(h)d_{Y}(h)+\sum_{h \in T} d_*(h)d_{Y}(h)\\  \label{eq:wall-nonlinear-contribution}
    \leq &\, R_0^2N^2+(\#T)N^2\leq 2R_0^2N^2.
\end{align}

Let $d_{r}(h)$ count the number of remaining pairs, that is, the number of $(w_i,w_j)\in W\times W$ with $h=w_i-w_j$ such that the two assigned components $C_i, C_j$ with $w_i\in C_i, w_j\in C_j$ are both affine lines.
Then
\[
    \sum_h d_{r}(h)=(\#W_{\mathrm{line}})^2.
\]
Since $d_{Y}(h)\leq N$ for each $h$, by \eqref{eq: wall line points}, we obtain
\begin{equation}
\label{eq:wall-line-contribution}
    \sum_h d_{r}(h)d_{Y}(h) \leq N(\#W_{\mathrm{line}})^2
    \leq R_0\Lambda_{ F}(W) N^2.
\end{equation}
Note that
\[
    E(W,Y)=\sum_h\bigl(d_*(h)+d_{r}(h)\bigr)d_{Y}(h),
\]
Combining \eqref{eq:wall-nonlinear-contribution} and
\eqref{eq:wall-line-contribution}, and using $R_0\leq m_FD$, we prove \eqref{eq:wall-energy}.
\end{proof}

\bigskip

\section{Proof of the main theorem} 
\label{proof-section}

In this section we prove Theorem~\ref{main-thm}.
We first show that Theorem \ref{main-thm} reduces to the following special case.

\begin{proposition}
\label{main-prop}
Let $F:\ZR^3\to\ZR$ be a polynomial that is irreducible over $\ZR$ and  $\deg_z F\ge 1$.
Let $\Sigma\subset Z(F)$ be the graph of a function of $\vec{u}$, where $x=(\vec{u},z)\in\ZR^3$.
If $\deg F\geq2$, then for any $\e>0$, there exists a constant $C_{m_F,\e}$ depending only on the degree of $F$ such that for any finite $X\subset \Sigma$,
\begin{equation}
\label{prop-esti}
    E(X)\leq C_{m_F,\e}\Lambda_{F}(X)(\#X)^{2+\e}.
\end{equation}
\end{proposition}

\begin{proof}[Proposition \ref{main-prop} implies Theorem \ref{main-thm}]

For given $F$ and $X$ from the statement of Theorem \ref{main-thm}, apply Lemma~\ref{lem: graph cover} to obtain the decomposition
\begin{equation}
\label{graph-decomp}
    X=\bigsqcup_{i=1}^K X_i, 
\end{equation}
where each $X_i$ lies on one coordinate graph. 
As shown in \eqref{eq:linear-invariance}, linear transformations preserve difference equations, affine lines in $ Z(F)$, and the quantity $\Lambda_{F}(X)$.

Apply Proposition \ref{main-prop} to each $X_i$ in \eqref{graph-decomp} so that, since $\Lambda_{F}(X_i)\leq\Lambda_{F}(X)$,
\begin{align*}
    E(X)
    \ll_K\sum_{i=1}^K E(X_i)\ll_K\sum_{i=1}^K C_{m_F,\e}\Lambda_{F}(X_i)(\#X_i)^{2+\e}\ll_K C_{m_F,\e}\Lambda_{F}(X)(\#X)^{2+\e}.
\end{align*}
Since $K^4\leq m_F^4$, a conclusion from Lemma~\ref{lem: graph cover}, this proves \eqref{main-esti}. 
\qedhere

\end{proof}

\medskip

We prove Proposition \ref{main-prop} in the rest of the section.

\begin{proof}[Proof of Proposition \ref{main-prop}]

For $t\in\R^3$, define
\[
    \Gamma_t:= Z\bigl(F(x),F(x-t)\bigr)
\]
as in Lemma \ref{difference-curve-lem}. 
Also, recall \eqref{translation-stabilizer} that
\[
    \mathcal{T}_F = \{t\in\R^3 : F(x-t)=F(x) \text{ for every }x\},
\]
and that if $t\notin\mathcal T_F$, then from Lemma \ref{difference-curve-lem} we know that $\Gamma_t$ is a union of algebraic curves of total degree at most $m_F^2$. 
We remark that the constant $C_{m_F}$ may vary from line to line, but it always depends only on $m_F$.

\medskip

\noindent {\bf Step 1: planar polynomial partitioning and pullback.}

Fix an integer $D\geq2$. 
Let $\pi:\ZR^3\to\ZR^2$ be the projection onto the horizontal plane.
Apply Theorem~\ref{thm:planar-partition} to $\pi(X)$ to obtain
\begin{enumerate}
    \item A polynomial $q\in\R[u_1,u_2]$ with degree $\leq D$.
    \item A cell decomposition $\ZR^2=Z(q)\sqcup(\sqcup_i \Omega_i)$.
\end{enumerate}
Discard those $\Omega_j$ that contain no points in $\pi(X)$.
For the remaining cells $\Omega_j$, define 
\begin{equation}
\nonumber
    X_i=X\cap\pi^{-1}(\Omega_i),
    \qquad
    W=X\cap\pi^{-1} (Z(q)),
    \qquad
    Y=X\setminus W=\bigsqcup_iX_i.
\end{equation}
Since $\pi$ is one-to-one on $\Sigma$, the planar polynomial partitioning moreover gives
\begin{equation}
\label{eq:cell-size-count}
    \#X_i\leq\alpha (\#X)D^{-2},
    \qquad\text{and}\qquad
    \#\{i:X_i\neq\varnothing\}\leq\alpha D^2.
\end{equation}

\medskip

\noindent {\bf Step 2: the first crossing.}

For each pair of cell labels $(i,j)$, write $d_{ij}(t)=d_{X_i-X_j}(t)$. 
Lift the two planar walls to $\R^3$ by
\[
 Q_0(x)=q(\pi x),
 \qquad
 Q_t(x)=q(\pi(x-t)).
\]
Fix $t\notin\mathcal T_F$.  
By Lemma \ref{difference-curve-lem}, $\Gamma_t$ is a real algebraic set of dimension at most one, defined by polynomials of degree at most $m_F$.
Applying Theorem \ref{thm:curve-cutting} to $\Gamma_t$ and to the polynomials $Q_0,Q_t$, we know that
\begin{equation}
\label{eq:cut-spatial-difference-curve}
    \Gamma_t\setminus\bigl( Z(Q_0)\cup Z(Q_t)\bigr)
\end{equation}
has at most $C_{m_F}D$ semialgebraically connected components.

Suppose that $d_{ij}(t)>0$, so there exists $x\in X_i$ such that $x-t\in X_j$, which implies that $x\in\Gamma_t$ and both $Q_0(x)$ and $Q_t(x)$ are nonzero. 
On each connected component of \eqref{eq:cut-spatial-difference-curve}, the continuous maps
$x\mapsto\pi x$ and $x\mapsto\pi(x-t)$ remain in fixed connected components of $\R^2\setminus Z(q)$.  
Thus one such connected component determines at most one ordered pair of cell labels $(i,j)$, yielding
\begin{equation}
\label{pairs-count}
    \#\{(i,j):d_{ij}(t)>0\}\leq C_{m_F}D.
\end{equation}
This implies that for $t\notin\mathcal T_F$, we have
\begin{equation}
\nonumber
    d_Y(t)^2 =\Big(\sum_{i,j}d_{ij}(t)\Big)^2\leq C_{m_F}D\sum_{i,j}d_{ij}(t)^2.
\end{equation}
By \eqref{eq:stabilizer-energy} and the fact that $\Lambda_{F}(Y)\leq\Lambda_F(X)$, summing over $t$ yields
\begin{equation}
\label{eq:first-crossing-energy}
    E(Y)\leq\Lambda_{F}(X) (\#X)^2
    +C_{m_F}D\sum_{i,j}\sum_t d_{ij}(t)^2.
\end{equation}

\medskip

\noindent {\bf Step 3: the second crossing.}

For every cell, write $d_i(u)=d_{X_i}(u)$. 
The identity \eqref{useful-identity} gives
\begin{equation}
\label{identity-proof}
    \sum_{i,j}\sum_t d_{ij}(t)^2=\sum_u\Big(\sum_i d_i(u)\Big)^2.
\end{equation}
For $u\notin\mathcal T_F$, by taking $i=j$ in \eqref{pairs-count}, we have
\begin{equation}
\nonumber
    \#\{i:d_i(u)>0\}\leq C_{m_F}D,
\end{equation}
which gives
\begin{equation}
\label{not-in-cT}
    \Big(\sum_i d_i(u)\Big)^2\leq C_{m_F}D\sum_i d_i(u)^2, \qquad u\not\in \cT_F. 
\end{equation}
For $u\in\mathcal T_F$, one has
$\sum_i d_i(u)\leq d_Y(u)$.  
Therefore, by \eqref{eq:stabilizer-energy} and the fact that $\Lambda_{F}(Y)\leq\Lambda_F(X)$ and $\#Y\leq \#X$, we have
\begin{equation}
\label{eq:exceptional-second-crossing}
    \sum_{u\in\mathcal T_F}
    \Big(\sum_i d_i(u)\Big)^2
    \leq\sum_{u\in\mathcal T_F}d_Y(u)^2
    \leq\Lambda_{F}(X)(\#X)^2.
\end{equation}
Combine \eqref{identity-proof}, \eqref{not-in-cT}, and
\eqref{eq:exceptional-second-crossing} so that
\begin{equation}
\nonumber
    \sum_{i,j}\sum_t d_{ij}(t)^2\leq\Lambda_{F}(X)(\#X)^2+C_{m_F}D\sum_iE(X_i).
\end{equation}
Substituting this into \eqref{eq:first-crossing-energy} gives
\begin{equation}
\label{eq:off-wall}
    E(Y)
    \leq C_{m_F}\Lambda_{F}(X)D^2(\#X)^2
    +C_{m_F}D^2\sum_iE(X_i).
\end{equation}

\medskip

\noindent {\bf Step 4: Iteration and finishing the proof.}

Apply Lemma \ref{lem: wall energy} to the polynomial $q$ (and $F$ and $\Sigma$ in Proposition \ref{main-prop}) from Step 1.
Since $W\subseteq X$, one has $\Lambda_{F}(W)\leq\Lambda_F(X)$ and $\#W\leq \# X$.
Thus, for any finite $V\subset \ZR^3$ with $\#V\leq \#X$, 
\begin{equation}
\label{eq:wall-mixed-ready}
    E(W,V)\leq C_{m_F}\Lambda_{F}(X)D^2(\#X)^2.
\end{equation}

The disjoint decomposition $X=Y\sqcup W$ gives the pointwise estimate
\begin{equation}
\nonumber
    d_X(t)=d_Y(t)+d_{Y-W}(t)+d_{W-Y}(t)+d_W(t).
\end{equation}
Using the triangle inequality, and \eqref{eq:wall-mixed-ready} with $V=Y$ and $V=W$, we have
\begin{equation}
\nonumber
    E(X)\ll E(Y)+C_{m_F}\Lambda_F(X)D^2(\#X)^2.
\end{equation}
By \eqref{eq:off-wall}, we conclude that
\begin{equation}
\nonumber
    E(X)\leq C_{m_F}\Lambda_F(X)D^2(\#X)^2+C_{m_F}D^2\sum_iE(X_i).
\end{equation}
Together with \eqref{eq:cell-size-count}, it implies the useful iteration
form
\begin{equation}
\label{eq:recurrence}
    E(X)\leq C_{m_F}\Lambda_F(X) D^2(\#X)^2+C_{m_F}\alpha D^4 \sup_{\substack{X'\subset X\\ \#X'\leq\alpha (\#X)D^{-2}}}E(X'),
\end{equation}
as $\Lambda_{F}(X')\leq\Lambda_F(X)$ for every $X'\subset X$.

\medskip

Finally, we fix $\e>0$, and let $C_0=C_0(m_F)$ be a constant for which
\eqref{eq:recurrence} holds.  
Choose an integer $D=D(m_F,\e)\geq2$ and $C_{m_F,\e}$ sufficiently large such that
\begin{equation}
\nonumber
    \alpha D^{-2}\leq \frac{1}{2}, \qquad C_0\alpha^{3+\e}D^{-2\e}\leq\frac12,\qquad C_{m_F,\e}\geq2C_0D^2.
\end{equation}
We prove \eqref{prop-esti} by induction on $N=(\#X)$.  
The base case $N=1$ is clear.

Let $N\geq2$ and assume that \eqref{prop-esti} holds for every finite subset of $\Sigma$ of size $\leq N/2$. 
Since $\Lambda_{F}(X')\leq\Lambda_F(X)$ for all $X'\subset X$, using induction to handle $E(X')$ in \eqref{eq:recurrence}, 
\begin{align}
\nonumber
    E(X)\,&\leq C_0\Lambda_F(X) D^2N^2+C_0\alpha D^4\cdot C_{m_F,\e}\Lambda_F(X)(\alpha N/D^2)^{2+\e}\\ \nonumber
    &=(C_0C_{m_F,\e}^{-1}D^{2}N^{-\e}+C_0\alpha^{3+\e}D^{-2\e}) C_{m_F,\e}\Lambda_F(X)N^{2+\e}
\end{align}
Note that our choice of parameters guarantees  $C_0C_{m_F,\e}^{-1}D^{2}N^{-\e}+C_0\alpha^{3+\e}D^{-2\e}\leq1$, which closes the induction and proves \eqref{prop-esti}.
\qedhere

\end{proof}

\bigskip

\bibliographystyle{alpha}
\bibliography{bibli}

\end{document}